\documentclass{amsart}
\usepackage{amsthm}
\usepackage{amssymb}
\usepackage{amsrefs}
\usepackage{enumerate}
\usepackage[matrix,arrow]{xy}
\usepackage{comment}

\numberwithin{equation}{section}
\theoremstyle{plain}
\newtheorem{thm}[equation]{Theorem}
\newtheorem{cor}[equation]{Corollary}
\newtheorem{prop}[equation]{Proposition}
\newtheorem{lem}[equation]{Lemma}
\newtheorem{hyp}[equation]{Hypothesis}
\newtheorem{conj}[equation]{Conjecture}

\theoremstyle{definition}
\newtheorem{defn}[equation]{Definition}
\newtheorem{rem}[equation]{Remark}

\DeclareMathOperator{\Hom}{Hom}
\DeclareMathOperator{\End}{End}
\DeclareMathOperator{\Irr}{Irr}
\DeclareMathOperator{\Stab}{Stab}
\DeclareMathOperator{\Lie}{Lie}
\DeclareMathOperator{\Gal}{Gal}

\newcommand{\mult}{^\times}
\newcommand{\bdd}{_{\mathrm{b}}}
\newcommand{\spl}{_{\mathrm{s}}}
\newcommand{\an}{_{\mathrm{an}}}
\newcommand{\nonram}{_{\mathrm{nr}}}
\newcommand{\tame}{^{\mathrm{t}}}
\newcommand{\der}{{\mathrm{der}}}
\newcommand{\HH}{{\mathcal{H}}}
\newcommand{\BBB}{{\mathcal{B}}}
\newcommand{\BB}{{\mathfrak{B}}}
\newcommand{\ZZ}{{\mathfrak{Z}}}
\newcommand{\fff}{{\mathfrak{f}}}
\newcommand{\sss}{{\mathfrak{s}}}
\newcommand{\ttt}{{\mathfrak{t}}}
\newcommand{\C}{{\mathbb{C}}}
\newcommand{\R}{{\mathbb{R}}}
\newcommand{\Q}{{\mathbb{Q}}}
\newcommand{\Rep}{{\mathcal{R}}}

\newcommand{\zGF}{{{}^\circ G(F)}}
\newcommand{\lsup}{{}^}
\newcommand{\doubleslash}{/\!\!/}

\begin{document}

\title{Regular Bernstein Blocks}
\author{Jeffrey D. Adler}
\address{
Department of Mathematics and Statistics \\
American University \\
4400 Massachusetts Ave NW \\
Washington, DC  20016-8050 \\
USA
}
\email{jadler@american.edu}

\author{Manish Mishra}
\thanks{The second named author was partially supported by SERB MATRICS 
and SERB ERCA grants}
\address{
Department of Mathematics \\
Indian Institute for Science Education and Research \\
Dr.\ Homi Bhabha Road, Pashan \\
Pune 411 008 \\
India 
}
\email{manish@iiserpune.ac.in}

\begin{abstract}
For a connected
reductive group $G$ defined over a non-archimedean local field $F$,
we consider the Bernstein blocks in the category of smooth representations
of $G(F)$.
Bernstein blocks whose cuspidal support involves a regular supercuspidal
representation are called \emph{regular} Bernstein blocks.
Most Bernstein
blocks are regular when the residual characteristic of $F$
is not too small.
Under mild hypotheses on the residual characteristic,
we show that the Bernstein 
center of a regular Bernstein block of $G(F)$
is isomorphic to the Bernstein center of a regular depth-zero Bernstein block
of $G^{0}(F)$, where $G^{0}$ is a certain twisted Levi subgroup of
$G$.
In some cases, we show that the blocks themselves are equivalent,
and as a consequence we prove the ABPS Conjecture in some new cases.
\end{abstract}

\maketitle

\section{Introduction}

\subsection{The problem}
Let $F$ be a nonarchimedean local field of residual characteristic $p$,
and $G$ a connected reductive $F$-group.
One wants to understand the category $\Rep(G(F))$
of smooth, complex representations of $G(F)$.
The Bernstein decomposition
\cite{bernstein:center}
allows us to decompose this category into a product
$\prod_{\{[M,\sigma]_G\}}{\Rep^{[M,\sigma]_G}(G(F))}$
of full subcategories
$\Rep^{[M,\sigma]_G}(G(F))$,
called \emph{Bernstein blocks},
so it is in some sense enough to understand each block.
(Here,
$[M,\sigma]_G$ ranges over the set of \emph{inertial classes}
of $G(F)$.
Such a class is represented by an $F$-Levi subgroup $M$ of $G$ and
an irreducible, supercuspidal representation $\sigma$ of $M(F)$,
and consists of the set of all pairs $(M'(F),\sigma')$ that are equivalent
to $(M(F),\sigma)$ under twisting by unramified characters of $M(F)$
and conjugation by elements of $G(F)$.)
Each block $\Rep^{[M,\sigma]_G}(G(F))$ is a module category over
an algebra,
and one can construct an appropriate Hecke algebra
in a more-or-less explicit
way if $[M,\sigma]_G$ has an associated Bushnell-Kutzko \emph{type}
\cite{bushnell98}.
In particular, if $\sigma$ has depth zero,
then the structure of the Hecke algebra is given by Morris \cite{Morris93}.
Thus, in principle, one way to understand the category $\Rep(G(F))$
would be to show that for each block $\Rep^{[M,\sigma]_G}(G(F))$,
there is another connected reductive $F$-group $G^0$
and a depth-zero block $[M^0,\sigma^0]_{G^0}$ for $G^0(F)$
such that the blocks are equivalent as categories.

Our aim in this paper is to study the structure of Bernstein blocks using the above 
approach in many cases.

Let us remark that Kim and Yu \cite{KY} give a construction
of types (see \cite{KY}*{\S 7.4} for more details),
and Fintzen \cite{fintzen:types-tame}
has recently shown that if $G$ splits over a tame
extension of $F$ and $p$ does not divide the order of the Weyl
group of $G$, then every Bernstein block for $G(F)$ admits
a type arising from this construction.
The Kim-Yu construction starts with a certain datum $\Sigma$, out of which it 
produces a sequence of types $(K^{i},\rho_{i})$ for groups $G^{i}(F)$,
$0\leq i\leq d$, where $\vec{G}=(G^{0}\subsetneq\ldots\subsetneq G^{d}=G)$
is a tower of twisted Levi subgroups of $G$ and which is part of the datum $\Sigma$. Write $(K,\rho)$ for
$(K^{d},\rho_{d})$ and write $\HH(G,\rho^{\vee})$
(resp.\ $\HH(G^{0},\rho_{0}^\vee)$)
for the Hecke algebra associated to the type $(K,\rho)$
(resp.\ $(K^{0},\rho_{0})$). Denote by $\mathcal{R}_\mathrm{f}(G,\rho)$ (resp. $\mathcal{R}_\mathrm{f}(G^0,\rho_0)$),
the category of finite-length modules over
$\HH(G,\rho^{\vee})$ (resp.\ $\HH(G^{0},\rho_{0}^\vee)$).
Existing results
\cites{howe-moy:harish-chandra,moy:u21,moy:gsp4,Kim2001,Heiermann2017},
together with a conjecture of Yu \cite{Yu00}*{Conjecture 0.2}) on supercuspidal
blocks, suggest a more general conjecture:

\begin{conj}\label{conj-Hecke}
There exists a bijection, determined by the choice of $\Sigma$, between the simple modules of 
$\HH(G,\rho^{\vee})$ and those of $\HH(G^{0},\rho_{0}^\vee)$.
\end{conj}

While the Hecke algebra $\HH(G^0,\rho_0^\vee)$ is determined
up to isomorphism,
it (and thus the bijection of the Conjecture) depends on $\Sigma$,
which is not uniquely determined.
\begin{rem} \label{rem-Hecke}
In many cases, something stronger than Conjecture \ref{conj-Hecke} holds.
Namely, the datum $\Sigma$ determines an isomorphism
of Hecke algebras 
\begin{equation}\label{eq-Hecke}
\HH(G,\rho^{\vee}) \overset{\sim}{\longrightarrow} \HH(G^{0},\rho_{0}^\vee).
\end{equation}
This is however not true in general.
In many cases where the isomorphism of Hecke algebras fails,
it may still be true that there exists an equivalence of categories
\begin{equation}\label{eq-R_f}
\mathcal{R}_\mathrm{f}(G,\rho) \cong \mathcal{R}_\mathrm{f}(G^0,\rho_0),
\end{equation}
determined by the choice of $\Sigma$.
\end{rem}

\subsection{Bernstein center}

In \cite{kaletha:regular-sc}, Kaletha studies a large class of superpercuspidal representations
which he calls \emph{regular}.
Most supercuspidal representations are of this kind when
$p$ is not too small
(see the paragraph following \cite{kaletha:regular-sc}*{Definition 3.7.3} 
to understand what is meant by ``most'').
We call a Bernstein block $\Rep^{[M,\sigma]_G}(G(F))$
\emph{regular} if $\sigma$ is a regular supercuspidal representation.
Under certain restrictions on $p$,
we show in Theorem \ref{thm:b-center-isom} that 
\begin{thm}\label{mm-thm}
The Bernstein center of a regular Bernstein block $\Rep^{[M,\sigma]_G}(G(F))$ is isomorphic to
the Bernstein center of a depth-zero regular Bernstein block \\ ${\Rep^{[M^{0},\sigma_{0}]_{G^0}}(G^0(F))}$ of a
twisted Levi subgroup $G^0$ of $G$. 
\end{thm}

When $M=G$, this result is covered by \cite{mishra19}*{Theorem 6.1} which proves such an 
isomorphism for tame supercuspidal blocks.
Thus, Theorem \ref{mm-thm} is a generalization
for regular Bernstein blocks 
of the main result of \cite{mishra19}.

\subsection{Hecke algebra isomorphism}

Suppose $M=G$. Then Conjecture \ref{conj-Hecke} would follow from
\cite{Yu00}*{Conjecture 0.2},
which appears here as Conjecture \ref{conj:yu}.
We show (Corollary \ref{cor:yu-true}) that the conjecture is true
if the $F$-split rank of the center of $G$ is at most $1$;
or if $\sigma$ is generic;
or if $p$ satisfies some mild hypotheses,
and the restriction of $\sigma$ to the derived group of $G$
decomposes into regular supercuspidal representations.

Now let us consider 
the case where $M$ is not necessarily equal to $G$.
Let $(K,\rho)$ be the $[M,\sigma]_G$-type obtained out of
the datum $\Sigma$ by the Kim-Yu construction.
Assume that the tower $\vec{G}$
of twisted Levi subgroups contained in the datum $\Sigma$ 
consists of $F$-Levi subgroups.
Then we show
(in Corollary \ref{P-types})
that Equation (\ref{eq-Hecke}) and consequently Conjecture \ref{conj-Hecke} holds. 
This follows as an easy consequence of the results in \cite{KY}.
The condition that the tower $\vec{G}$ consists of $F$-Levi subgroups 
holds for instance when the connected center of $G^0$ is split modulo the center of $G$. 
In particular, it holds when $G$ is split and $M$ is a maximal split torus
in $G$. Thus, as a very special case, one obtains Equation (\ref{eq-Hecke})
(and consequently Conjecture \ref{conj-Hecke}) for principal series blocks 
of split groups, recovering the result of Roche \cite{roche:2000}. We also prove 
Equation (\ref{eq-Hecke}) for certain principal series blocks of non-split groups (Proposition \ref{prop:prin-series}) for suitably large $p$. 

We prove weaker versions of Equation (\ref{eq-Hecke})
for more general situations in \S\ref{sec:partial}.

\subsection{Consequences for the ABPS Conjecture}
Aubert, Baum, Plymen, and Solleveld \cite{abps17} conjecture
that the set of irreducible objects in a Bernstein block
has the structure of a particular \emph{twisted extended quotient}.
This ``ABPS Conjecture''
holds for depth-zero Bernstein blocks due to the work of Solleveld \cite{solleveld2012}.
We show, using this result of Solleveld, that under certain restrictions
on $p$,
Conjecture \ref{conj-Hecke} for a regular Bernstein block implies
ABPS Conjecture for that block.
Consequently, we obtain the ABPS Conjecture
for many new cases (Theorem \ref{thm:ABPS-Levi}),
namely those mentioned in the previous section.
In particular, as the simplest case,
we obtain ABPS Conjecture for principal series blocks of split groups.
This reproves a part of the main result in \cite{abps17:p-series} in a much shorter way. 
However, unlike Theorem \ref{thm:ABPS-Levi}, the proof in loc.~cit. is independent 
of the main result of \cite{solleveld2012}. 

Solleveld has, in a recent preprint, given a complete proof of the ABPS Conjecture \cite{Sol2020}*{Theorem D}. His proof bypasses the type-theoretic approach and builds instead on the ideas of Bernstein \cite{bernstein:center} and Heiermann \cite{Hei2011}. Moreover, under certain technical hypotheses (the existence and compatibility of certain co-cycles), he also shows \cite{Sol2020}*{Theorem B} that the category $\mathcal{R}_\mathrm{f}(G,\rho)$ is equivalent to the finite-length module category of a twisted affine Hecke algebra. Under the same hypothesis and some additional residue characteristic hypotheses (Hypothesis \ref{hyp:p}), his results may be used to establish Equation (\ref{eq-R_f}) for regular Bernstein blocks provided one shows that the isomorphisms of Theorems \ref{mm19} and \ref{thm:weyl-isom} respect a certain finer structure dictated by Harish-Chandra $\mu$-function. We plan to pursue these matters in a subsequent work.


\subsection{Acknowledgments}
It is a pleasure for the authors to thank the following.
Maarten Solleveld clarified for us
the definition of the twisted extended quotient and pointed out an error
in a previous draft.
Anne-Marie Aubert pointed out that the isomorphism of Hecke algebras in Remark \ref{rem-Hecke} is false in general for tame Bernstein blocks.
Ad\'ele Bourgeois pointed out a typographical error in Definition \ref{def:reg}. 
An anonymous referee provided comments that helped us to improve the exposition
of this work.

\section{Notation}

Throughout this article, $F$ denotes a non-archimedean local field
of residue characteristic $p$.
If $G$ is a reductive $F$-group and $K$ is a subgroup of $G(F)$,
we will denote $gKg^{-1}$ by $\lsup gK$ for $g\in G(F)$. If $\rho$
is a representation of $K$, then $\lsup g\rho$ will denote the representation
$x\mapsto\rho(g^{-1}xg)$ of $\lsup gK$. If $H$ is a subgroup
of a group $G$, then $N_G(H)$ denotes
the normalizer of $H$ in $G$. 

Write $\BBB(G,F)$ for the Bruhat-Tits building of $G(F)$. 
For any point $x\in\BBB(G,F)$,
we write $G(F)_{x,0}$ for the parahoric subgroup of $G(F)$ associated to $x$,
and we let $G(F)_{x:0:0+}$ denote the quotient of $G(F)_{x,0}$ by
its maximal pro-$p$ subgroup.
If $G$ is a torus, then we omit $x$ from the notation
and just write $G(F)_0$ for the unique parahoric subgroup of $G(F)$,
and
we also let $G(F)\bdd$ denote
the unique maximal bounded subgroup.

For an $F$-group $A$ whose connected part is a torus,
let $A\spl$ and $A\an$ denote the maximal $F$-split
and $F$-anisotropic subtori in $A$, respectively.

\section{Review of the Bernstein center}

Throughout, $G$ denotes a connected reductive $F$-group.
We quickly review general theory and fix some notation. 

\subsection{Bernstein decomposition}

Let $\zGF$ denote the subgroup
\[
\{g\in G(F)\colon|\nu(g)|=1\:\text{for every character \ensuremath{\nu\in\mathrm{Hom}}}_{F}(G,\mathbb{G}_{\mathrm{m}})\},
\]
and let $X\nonram(G(F))=\Hom(G(F)/\zGF,\C\mult)$.
The group $X\nonram(G(F))$ is called the group of
\emph{unramified characters} of $G(F)$. Consider the pairs $(L(F),\sigma)$
where $L$ is an $F$-Levi subgroup of $G$
and $\sigma$ is an irreducible supercuspidal representation of $L(F)$.
Denote by $[L,\sigma]_{G}$ its
\emph{inertial equivalence class}:
the set of pairs $(L',\sigma')$,
such that
$L'$ is an $F$-Levi subgroup of $G$, $\sigma'$ is an irreducible,
supercuspidal representation of $L(F)$,
and $(L',\sigma')=(\lsup g L, \lsup g \sigma\otimes \nu)$
for some $g\in G(F)$ and $\nu\in X\nonram(L(F))$.
The set of inertial equivalence
classes is called the \emph{Bernstein spectrum} of $G(F)$,
which we denote by $\BB(G,F)$.
Let $\Rep(G(F))$ denote the
category of smooth representations of $G(F)$.
We say that a smooth
irreducible representation $\pi$ has \emph{inertial support}
$\sss=[L,\sigma]_{G}$
if $\pi$ appears as a subquotient of a representation parabolically
induced from some element of the class $\sss$. Define a full
subcategory $\Rep^\sss(G(F))$ of $\Rep(G(F))$
as follows: $\pi\in\Rep(G(F))$ belongs to $\Rep^\sss(G(F))$
if each irreducible subquotient of $\pi$ has inertial support $\sss$.
We will denote the class of irreducible objects in $\Rep^\sss(G(F))$
by $\Irr^\sss(G(F))$. 
\begin{thm}[Bernstein]
We have a decomposition
\[
\Rep(G(F))=\prod_{\sss\in\BB(G,F)}\Rep^\sss(G(F)).
\]
\end{thm}

\subsection{Hecke algebra}

Let $(\tau,W)$ be an irreducible representation of a compact open
subgroup $J$ of $G(F)$. The Hecke algebra $\HH(G,\tau)$
is the space of compactly supported functions
$f:G(F)\longrightarrow\End_\C(\tau)$
satisfying 
\[
f(j_{1}gj_{2})=\tau(j_{1})f(g)\tau(j_{2}).
\]
The standard convolution algebra gives $\HH(G,\tau)$ the
structure of an associative $\C$-algebra with identity. 

We say that $(J,\tau)$ is an \emph{$\sss$-type} if $\Rep^\sss(G(F))$
is precisely the subcategory of $\Rep(G(F))$ consisting of smooth
representations which are generated by their $\tau$-isotypic component.
In that case, $\Rep^\sss(G(F))$ is equivalent to the
category of non-degenerate modules of $\HH(G,\tau^\vee)$, 
where $\tau^\vee$ is the dual of $\tau$ (\cite{bushnell98}).

\subsection{Bernstein center}

Let $\sss=[L,\sigma]_{G}\in\BB(G,F)$ and set
$\ttt=[L,\sigma]_{L}\in\BB(L,F)$.
The center
$\mathfrak{Z}(G(F))$ (resp.\ $\mathfrak{Z}^{\sss}(G(F))$, resp.\ $\mathfrak{Z}^{\ttt}(L(F))$)
of the category $\Rep(G(F))$ (resp.\ $\Rep^{\sss}(G(F))$,
resp.\ $\Rep^{\ttt}(L(F))$) is called its
\emph{Bernstein center}.
Recall here that the \emph{center}
of an abelian category is the endomorphism ring of the identity functor. 
Let
\[
N^\sss
=
\{
n\in N_{G}(L)(F)
\mid
\lsup n \sigma\cong\sigma\nu \:\text{for some $\nu\in X\nonram(L(F))$}
\},
\]
and write $W^\sss=N^\sss/L(F)$. The set $\Irr^\ttt(L(F))$ of irreducible objects of $\Rep^{\ttt}(L(F))$
has the structure of a complex affine variety, on which the group $W^\sss$
acts algebraically.
(The variety $\Irr^\ttt(L(F))$
is (non-canonically) isomorphic to the quotient
$X\nonram(L(F))/ \{\chi\in X\nonram(L(F))\mid\sigma\chi=\sigma\}$.)

\begin{thm}[Bernstein]\cite{bernstein:center}
The Bernstein center of $\Rep^\sss(G(F))$ can be viewed as the ring of regular
functions on the quotient variety $\Irr^\ttt(L(F))/W^\sss$. 
\end{thm}

\section{Review of tame and regular supercuspidal representations}

\subsection{Tame supercuspidal representations}

Assume that $G$ splits over
the maximal tamely ramified extension $F\tame$ of $F$.
A tamely twisted Levi subgroup of $G'$ of $G$
is a reductive $F$-subgroup such that $G'\otimes F\tame$
is an $F\tame$-Levi subgroup of $G\otimes F\tame$.
Let $\Sigma$ denote the datum
$(\vec{G},\pi_{0},\vec{\phi})$,
where $\vec{G}=(G^{0},\ldots,G^{d})$ is a tower
of algebraic subgroups of $G$, 
\[
G^{0}\subsetneq\ldots\subsetneq G^{d}=G,
\]
such that $Z(G^{0})/Z(G)$ is anisotropic and each $G^{i}$ is a
tamely twisted Levi subgroup of $G$, $\pi_{0}$ is a depth-zero supercuspidal
representation of $G^{0}(F)$ and
$\vec{\phi}=(\phi_{0},\ldots,\phi_{d})$
is such that $\phi_{i}:G^{i}(F)\longrightarrow\C\mult$ is
a smooth character of depth $r_{i}>0$.
We require the datum $\Sigma$ to satisfy
several technical conditions.
(See \cite{HM08}*{\S3.1} for a precise list.
Note that this source and others,
e.g., \cite{kaletha:regular-sc},
use the notation $\pi_{-1}$ for the representation that we are calling
$\pi_0$.)

Yu's construction \cite{Yu00} produces a supercuspidal representation
$\pi$ out of the datum $\Sigma$.
Recently, Fintzen \cite{fintzen19} 
has shown that if the residue characteristic $p$
does not divide the order of the Weyl group, this construction
yields all supercuspidal representations of $G(F)$.

\subsection{Regular supercuspidal representations}

Let $(S,\theta)$ be a pair consisting of tame maximal $F$-torus $S$
in $G$
and a character
$\theta:S(F)\longrightarrow\C\mult$.
A Kaletha-Howe
factorization of $(S,\theta)$ is a pair
$(\vec{G},\vec{\phi})$,
where
$\vec{G}=(S=G^{-1}\subset G^{0}\subsetneq\ldots\subsetneq G^{d}=G)$
is a tower of tamely twisted Levi $F$-subgroups
and $\vec{\phi}=(\phi_{-1},\ldots,\phi_{d})$
is a sequence of characters $\phi_{i}:G^{i}(F)\longrightarrow\C\mult$
and satisfying 
\[
\theta=\prod_{i=-1}^{d}\phi_{i}|_{S(F)}
\]
and some additional technical conditions
(see \cite{kaletha:regular-sc}*{Definition 3.6.2}).

If the datum $(S,\theta)$ is a \emph{tame regular elliptic pair},
i.e.,
it satisfies the conditions in \cite{kaletha:regular-sc}*{Definition 3.7.5},
then the pair $(S,\phi_{-1})$ determines a depth-zero supercuspidal
representation $\pi_{0}$ of $G^{0}(F)$ and the datum 
\[
\Sigma=
\left(
(G^{0}\subsetneq\ldots \subsetneq G^{d}=G),\pi_{0},(\phi_{0},\ldots,\phi_{d})
\right)
\]
is a Yu datum and therefore produces a supercuspidal representation
$\pi(S,\theta)$ of $G(F)$. These representations are called \textit{regular}
supercuspidal representations. 

If $(S,\theta)$ and $(S',\theta')$ are tame regular
elliptic pairs in $G$, then $\pi(S,\theta)$ and $\pi(S',\theta')$
are equivalent if and only if
$(S,\theta)$ and $(S',\theta')$
are conjugate in $G(F)$.

\section{A result on multiplicity one upon restriction}

In this section, we will assume the following.

\begin{hyp}
\label{hyp:p}
The residue characteristic $p$ of $F$ is not a \emph{bad prime}
for $G$
(see \cite{kaletha:regular-sc}*{\S 2.1})
and does
not divide the order of the fundamental group of the derived group $G^\der$
of $G$.
\end{hyp}

Let $(S,\theta)$ be a tame regular elliptic pair in $G$ and $\pi(S,\theta)$ the associated regular 
supercuspidal representation.
Let $G'$ be an $F$-subgroup of $G$ such that $G^\der\subset G'\subset G$.
Write $S'=S\cap G'$
and $\theta'=\theta|_{S'(F)}$.
Let
$N_{G(F)}(S',\theta'|_{S'(F)_{0}})$
denote
the stabilizer of $\theta'|_{S'(F)_{0}}$
in $N_{G}(S')(F)$.

\begin{defn}\label{def:reg}
The pair $(S,\theta)$ is \emph{regular in $G'$}
if $N_{G(F)}(S',\theta'|_{S'(F)_{0}}) = S(F)$.
\end{defn}

\begin{thm}
\label{thm:v-regular-MF}
If $(S,\theta)$ is a tame regular elliptic
pair in $G$ which is regular in $G'$,
then $\pi(S,\theta)|_{G'(F)}$
is multiplicity free. 
\end{thm}

We observe that $(S,\theta)$ is regular in $G'$ if and only if
the components of $\pi|_{G'(F)}$ are regular supercuspidals.
If $(S,\theta)$ is not regular in $G'$, then higher multiplicities
can occur, even though $(S,\theta)$ is regular in $G$.
See \cite{adler-dprasad:decomposition}*{\S7} for an example.

Looking at general irreducible representations 
of $G(F)$, there are many situations where one can prove
that the restriction to $G'(F)$ is multiplicity free.
See \cite{choiy:multiplicity-restriction}
for a conjecture, together with a proof
for tempered representation
under some assumptions about the local Langlands correspondence.
See \cite{adler-dprasad:decomposition} for a conjecture,
reduction to the tempered case,
examples, and an announcement of Theorem \ref{thm:v-regular-MF}.

\begin{proof}[Proof of Theorem \ref{thm:v-regular-MF}]
Since $(S,\theta)$ is regular in $G'$, it follows from
\cite{kaletha:regular-sc}*{Definition 3.7.5 and Lemma 3.6.5} that $(S',\theta')$
is a tame regular elliptic pair in $G'$.
Let $V'$ (resp.\ $V$)
be the space realizing the representation $\pi(S',\theta')$
(resp.\ $\pi(S,\theta)$).
One can describe these spaces explicitly, since each regular
supercuspidal representation is induced from a representation
of a compact open subgroup.
We now describe an embedding of $V'$
in $V$. 

From Yu's construction \cite{Yu00},
we see that
\begin{itemize}
\item
$\pi(S,\theta)$ is induced from a smooth irreducible representation
$\rho$ of an open subgroup $K$ of $G(F)$ that is compact modulo
the center of $G(F)$;
\item
$\pi(S',\theta')$ is similarly induced from $(K',\rho')$;
and
\item
we can realize $\rho'$ as an irreducible 
subrepresentation of $\rho|_{K'}$.
\end{itemize}
Let $W$ and $W'$ be the spaces realizing $\rho$ and $\rho'$.
Thus, $W'$ is a subspace of $W$.
We have 
\[
W=\bigoplus_{k\in K/\Stab_{K}(W')} {}^k W'
.
\]
Pick coset representatives $k_{1},\ldots,k_{n}$
of $K/\Stab_K W'$.
Let 
\[
\iota:\End(W')\longrightarrow\End(W)
\]
be given by
\[
\iota \colon \lambda'\mapsto\bigoplus_{i=1}^{n} \lsup{k_i}\lambda'.
\]
Then for all $k'\in K'$,
$\iota(\rho'(k'))=\rho(k')$.
Define an embedding $\iota:V'\longrightarrow V$
by setting, for $f'\in V'$, $\iota(f')$ to be the function $f\in V$
given by
\[
f(g)=\begin{cases}
0 & \text{if $g\notin G'(F)K$},\\
\iota(f'(g'))\rho(k) &
	\text{if $f=g'k$, where $g'\in G'(F),\: k\in K$}
\end{cases}.
\]
To see that the function $f = \iota(f')$ is well defined,
let $g_{1}'k_{1}=g_{2}'k_{2}$,
where $g_{1}',g_{2}'\in G'(F)$ and $k_{1},k_{2}\in K$.
Then
$g_{1}^{\prime-1}g_{2}'=k_{1}k_{2}^{-1}\in G'(F)\cap K=K'$.
We have 
\begin{align*}
\iota(f'(g_{2}')\rho(k_{2})
&= \iota(f'(g_{1}'g_{1}^{\prime-1}g_{2}')\rho(k_{2})\\
&= \iota(f'(g_{1}'))\rho(k_{1}k_{2}^{-1})\rho(k_{2})\\
&= \iota(f'(g_{1}'))\rho(k_{1}).
\end{align*}
This proves that $f$ is well defined.
Now since $S(F)\subset K$,
it follows that $S(F)$ preserves $\iota(V')$.
Thus 
\[
\Stab_{G(F)}\iota(V')\supseteq S(F)G'(F).
\]
By Clifford theory
\[
\pi(S,\theta)
=
\bigoplus_{g\in G(F)/{\Stab_{G(F)}\iota(V')}} \lsup g \pi(S',\theta').
\]
We claim that for $g\notin\Stab_{G(F)}\iota(V')$,
$\lsup g \pi(S',\theta')\ncong\pi(S',\theta')$.
For suppose that $\lsup g \pi(S',\theta')\cong\pi(S',\theta')$
for some $g\in G(F)$.
Equivalently, $(\lsup gS',\lsup g\theta')$
is $G'(F)$-conjugate to $(S',\theta')$.
So there exists a $g'\in G'(F)$ such that $g'g$
stabilizes $(S',\theta')$. So $g\in G'(F)N_{G(F)}(S',\theta')$.
But by hypothesis,
$N_{G(F)}(S',\theta'|_{S'(F)_{0}})\subset S(F)$.
It follows that $g\in\Stab_{G(F)}\iota(V')$. This
proves the claim. 

Thus the summands
$\lsup g\pi(S',\theta')$ for $g\in G(F)/\Stab_{G(F)}\iota(V')$
are non-isomorphic and consequently, $\pi(S,\theta)|_{G'(F)}$
is multiplicity free. 
\end{proof}

\section{Hecke algebra for supercuspidal blocks}

Let $(\pi ,V)$ denote a supercuspidal representation of $G(F)$.
Then $\pi |_\zGF$ is a finite direct sum of finite representations
of $\zGF$, all of which are $G(F)$-conjugate up to isomorphism.
Fix an irreducible component $(^{\circ}\pi ,{}^{\circ}V)$ of $\pi |_\zGF$.
Then $\End_{G(F)}(\mathrm{ind}_\zGF^{G(F)}{}^{\circ}\pi)$
is canonically isomorphic to the convolution algebra $\HH(G,{}^{\circ}\pi)$
of $^{\circ}\pi $-spherical functions (\cite{roche:BC}*{\S1.1.2})
and the Bernstein component $\Rep^\sss(G(F))$ is canonically
equivalent to the module category of $\HH(G,{}^{\circ}\pi^\vee )$,
where $\sss=[G,\pi ]_{G}$. 
\begin{thm}
\label{thm:roche}\cite{roche:BC}*{Prop.\ 1.6.1.2}
The Hecke algebra
$\HH(G,{}^{\circ}\pi )$
is commutative if and only if
$\pi |_\zGF$ is multiplicity free. 
\end{thm}

\begin{prop}\label{prop:MF-0G}
The restriction $\pi |_\zGF$
is multiplicity free if 
the split rank of the identity component of the center of $G$ is $0$ or $1$. 
\end{prop}
\begin{proof}
Let $Z$ denote the center of $G$.
Then in this situation,
$G(F)/(Z(F)\cdot\zGF)$ is cyclic. The result then follows from \cite{roche:BC}*{Remark 1.6.1.3}.
\end{proof}

Now let $\pi$ be
a tame supercuspidal representation of $G(F)$.
Then by Yu's construction, $\pi $ is compactly
induced from a representation $\rho$ of an open, compact mod center subgroup
$K$ of $G(F)$.
Let $^{\circ}K$ denote the
maximal compact subgroup $K\cap{}\zGF$ of $K$,
and let $^{\circ}\rho$
be any irreducible component of $\rho|_{^{\circ}K}$.
Then $(^{\circ}K,{}^{\circ}\rho$)
is an $\sss$-type where $\sss=[G,\pi]_{G}$
(\cite{bushnell98}*{Prop.\ 5.4}).
The representation $\pi $ is constructed
out of a depth-zero representation $\pi _{0}$
of a twisted Levi subgroup
$G^{0}$ of $G$ together with additional data. The representation
$\pi _{0}$ is compactly induced from a representation $\rho_{0}$
of a compact mod center subgroup $K^{0}$ of $G^{0}(F)$.
Define an $\sss_0$-type
$(^{\circ}K^0,{}^{\circ}\rho_{0}$) analogously,
where $\sss_{0}=[G^{0}(F),\pi_{0}]_{G^{0}(F)}$. 

Jiu-Kang Yu makes the
following conjecture \cite{Yu00}*{Conjecture 0.2}. 

\begin{conj}[Yu's conjecture]
\label{conj:yu}
%
There is an algebra isomorphism 
\[
\HH(G,^{\circ}\rho)\cong\HH(G^{0},{}^{\circ}\rho_{0}).
\]
%
\end{conj}

\begin{cor}
\label{cor:yu-true}
Yu's conjecture holds if any of the following conditions hold:
\begin{enumerate}[(a)]
\item
The split rank of the identity component of the center of $G$ is $0$ or $1$. 
\item
Hypothesis \ref{hyp:p} holds, 
and
the irreducible components of $\pi |_{G^\der(F)}$ are regular.
\item
$\pi$ is generic.
\end{enumerate}
\end{cor}

\begin{proof}
By \cite{mishra19}*{Cor.\ 6.3},
the centers of
$\HH(G,^{\circ}\rho)$ and $\HH(G^{0},{}^{\circ}\rho_{0})$
are isomorphic.
Thus, the conjecture is true if both of these Hecke algebras
are commutative.
From Theorem \ref{thm:roche}
it will be enough to show that $\pi$ and $\pi_0$ restrict
without multiplicity to $\zGF$ and $^\circ G^0(F)$, respectively.
Suppose condition (a) holds.
Then our restrictions are multiplicity free from
Proposition \ref{prop:MF-0G},
together with the fact that the centers of $G$ and $G^0$
have the same split rank.
Suppose condition (b) holds.
By \cite{kaletha:regular-sc}*{Lemma 3.6.5},
the components of 
$\pi_{0}|_{(G^{0}\cap G^\der)(F)}$
are regular.
Thus our restrictions are multiplicity free from
Theorem \ref{thm:v-regular-MF}.
Suppose condition (c) holds.
From a result of Stephen DeBacker and Cheng-Chiang Tsai,
$\pi_0$ is also generic,
and so our restrictions are multiplicity free from
\cite{roche:BC}*{Remark 1.6.1.3}.
\end{proof}

Note that in proving that Yu's conjecture holds under condition (c),
we have used a result that is not yet in
the literature.
Readers who are unhappy about this can strengthen condition (c) to assert
that both $\pi$ and $\pi_0$ are generic.

\section{Bernstein center of regular blocks}
\label{sec:B-center}

Assume Hypothesis \ref{hyp:p}.

Let $L$
be an $F$-Levi subgroup of $G$ and let $(S,\theta)$ be a tame regular
elliptic pair in $L$. Write $\pi = \pi(S,\theta)$ for the associated regular
supercuspidal representation of $L(F).$ The pair $(S,\theta)$ produces
a chain $(S=G^{-1}\subseteq G^{0}\varsubsetneq\ldots\varsubsetneq G^{d}=G)$
of twisted Levi subgroups of $G$. Write $L^{0}=L\cap G^{0}$. Let
$\phi_{i}:G^{i}(F)\longrightarrow\C\mult$, $-1\leq i\leq d$,
be the sequence of characters obtained from a Howe factorization
of $(S,\theta)$.
Then $\theta=\theta_{-}\theta_{+}$, where $\theta_{-}=\phi_{-1}|_{S(F)}$
and $\theta_{+}=\prod_{i=0}^{d}\phi_{i}|_{S(F)}$. It follows from
\cite{kaletha:regular-sc}*{Lemma 3.6.5} that $(S,\theta_{-})$
is a depth-zero tame regular elliptic pair for $L^0$.
Let $\pi_{0}(S,\theta_{-})$ denote
the associated depth-zero regular supercuspidal representation of $L^{0}(F)$. 

Write $\ttt=[L,\pi]_{L}$ and $\ttt_{0}=[L^{0},\pi_{0}]_{L^{0}}$.

\begin{thm}\label{mm19}
\cite{mishra19}*{Theorem 6.1}The map 
\[
\fff:\pi\otimes\nu\in\Irr^{\ttt}(L(F))\mapsto\pi_{0}\otimes(\nu|_{L^{0}(F)})\in\Irr^{\ttt_{0}}(L^{0}(F)),\quad\nu\in X\nonram(L(F)),
\]
 is an isomorphism of varieties. 
\end{thm}

For 
$\sss=[L,\pi]_{G}$, recall that we denote by $W^\sss$
the stabilizer of $\sss$ in $(N_{G}L)(F)/L(F)$, i.e., 
$W^\sss=N^\sss/L(F)$, where 
\[
N^\sss
=
\{
n\in N_{G}(L)(F)
\mid
\lsup n\pi\cong\pi\nu,\text{ for some }\nu\in X\nonram(L(F))
\}.
\]
Similarly define $\sss_{0},N^{\sss_{0}}$ and $W^{\sss_{0}}$
by replacing $L,\pi$ and $G$ in the above definition by $L^{0},\pi_{0}$
and $G^{0}$.
\begin{lem}
\label{lem:stab}The stabilizer of $\theta|_{S(F)\bdd}$ in $N_{G}(S)(F)$
lies in $N_{G^{0}}(S)(F)$ and equals the stabilizer of $\theta_{-}|_{S(F)\bdd}$
there. 
\end{lem}
\begin{proof}
The proof is identical to the proof of \cite{kaletha:regular-sc}*{Lemma 3.6.5}
when we replace $S(F)_{r}$ there with $S(F)\bdd$. 
\end{proof}
We observe here that $\theta_{+}$ is invariant under $N_{G}(S)(F)$.
This follows from the observation in the first paragraph of the proof
of \cite{kaletha:regular-sc}*{Lemma 3.6.5} that for $0\leq i\leq d$,
$\phi_{i}|_{S(F)}$ is invariant under $N_{G}(S)(F)$.

\begin{thm}
\label{thm:weyl-isom}
There is a group isomorphism
\[
\iota_{\pi}:W^\sss\cong W^{\sss_{0}}. 
\]
Moreover, the isomorphism $\fff$ is equivariant with respect to
the isomorphism $\iota_{\pi}$.
\end{thm}

We remark here that since $G^0$ depends on the choice of $(S,\theta)$,
so does the isomorphism $\iota_{\pi}$.

\begin{proof}
Take $n\in N^\sss$. So $n\in N_{G}(L)(F)$ and $\lsup n(S,\theta)$
is $L(F)$-conjugate to $(S,\theta\lambda)$ for some $\lambda\in X\nonram(S)$.
This implies that there exists an $l\in L(F)$ such that $l^{-1}n\in\Stab_{N_{G}(L)(F)}(\theta|_{S(F)\bdd})$.
By Lemma \ref{lem:stab}, $l^{-1}n\in\Stab_{N_{G^{0}}(L^{0})(F)}(\theta_{-}|_{S(F)\bdd})$.
Thus $\lsup {l^{-1}n}(S,\theta_{-})=(S,\theta_{-}\lambda)$ and hence get a map $W^\sss\longrightarrow W^{\sss_{0}}$
induced by the map $n\in N^\sss\longrightarrow l^{-1}n\in N^{\sss_{0}}$.
We claim that this map is an isomorphism. Suppose $n_{0}\in N^{\sss_{0}}$.
Then $n_{0}\in N_{G^{0}}(L^{0})(F)$ such that $^{l_{0}^{-1}n_{0}}(S,\theta_{-})=(S,\theta_{-}\lambda_{-})$
for some $l_0\in L^{0}(F)$ and some $\lambda_{-}\in X\nonram(S)$.
Thus $l_{0}^{-1}n_{0}\in\Stab_{N_{G^{0}}(S)(F)}(\theta_{-}|_{S(F)\bdd})$.
By Lemma \ref{lem:stab}, $l_{0}^{-1}n_{0}\in\Stab_{N_{G}(S)(F)}(\theta|_{S(F)\bdd})$.
We now show that $l_{0}^{-1}n_{0}\in N_{G}(L)(F)$. 

Recall that $Z(L^{0})/Z(L)$ is $F$-anisotropic. That is, 
\[
Z(L)\spl^{\circ}=Z(L^{0})\spl^{\circ}.
\]
Let $n=l_{0}^{-1}n_{0}$. Since $n\in N_{G^{0}}(L^{0})(F)$, $n$
preserves $Z(L^{0})$ and since $n$ is rational, $n$ preserves $Z(L^{0})\spl^{\circ}$.
Thus $n$ preserves $Z(L)\spl^{\circ}$ and therefore also
preserves $Z_{G}(Z(L)\spl^{\circ})=L$. 

Finally since $\theta_{+}$ is invariant under
$l_{0}^{-1}n_{0}$, $^{l_{0}^{-1}n_{0}}(S,\theta)=(S,\theta\lambda_{-})$. This completes the
proof. 
\end{proof}
\begin{thm}
\label{thm:b-center-isom}
There is an algebra isomorphism 
\[
\ZZ^\sss(G(F))\cong\ZZ^{\sss_{0}}(G^0(F)).
\]
\end{thm}
\begin{proof}
By Theorem \ref{mm19}, the map 
\[
\fff:\pi\otimes\nu\in\Irr^\ttt(L(F))\mapsto\pi_{0}\otimes(\nu|_{L^{0}(F)})\in\Irr^{\ttt_{0}}(L^{0}(F)),\: \nu\in X\nonram(L(F))
\]
is an isomorphism of varieties and by Theorem \ref{thm:weyl-isom}, it is equivariant with respect to the isomorphism $\iota_{\pi}$. Consequently, there is an
isomorphism of quotient varieties, 
\[
\Irr^\ttt(L(F))/W^\sss\cong\Irr^{\ttt_{0}}(L(F))/W^{\sss_{0}}.
\]
Since $\ZZ^\sss(G(F))$ (resp.\ $\ZZ^{\sss_{0}}(G^0(F))$)
is the ring of regular functions on the Bernstein variety
$\Irr^\ttt(L(F))/W^\sss$
(resp.\ $\Irr^{\ttt_{0}}(L(F))/W^{\sss_{0}}$),
the result follows. 
\end{proof}

\section{Hecke algebra of tame types }

Let $G$ be tamely ramified. 
Consider the datum 
\begin{equation}\label{eq:datum}
\Sigma=
(
	({\vec{G}},M^{0}),
	(y,\{\iota\}),
	\vec{r},
	(K_{M^{0}},\rho_{M^{0}}),
	\vec{\phi}
)
\end{equation}
as in \cite{KY}*{\S 7.2}. So $\vec{G}=(G^{0},\ldots,G^{d})$
is a tamely ramified twisted Levi sequence in $G$, $M^{0}$ is a
Levi subgroup of $G^{0},$ $\vec{r}=(r_{0},\ldots,r_{d})$
is a sequence of real numbers, $y$ is a point in $\BBB(M^{0},F)$,
$\{\iota\}$ is a commutative diagram 
\[
\xymatrix{
\BBB(G^{0},F)\ar@{^{(}->}[r]
	& \BBB(G^{1},F)\ar@{^{(}->}[r]
	& \cdots\ar@{^{(}->}[r]
	& \BBB(G^{d},F)
\\
\BBB(M^{0},F)\ar@{^{(}->}[r]\ar@{^{(}->}[u]
	& \BBB(M^{1},F)\ar@{^{(}->}[r]\ar@{^{(}->}[u]
	& \cdots\ar@{^{(}->}[r] 
	& \BBB(M^{d},F)\ar@{^{(}->}[u]
}
\]
where $M^{i}$ is the centralizer in $G^{i}$ of
$Z(M^0)\spl$.
The diagram $\{\iota\}$ is required to be
$\vec{s}$-generic relative to $y$
(see \cite{KY}*{\S 3.5, Definition}),
where $\vec{s}=(0,r_{0}/2,\ldots,r_{d-1}/2)$.
The group $K_{M^{0}}$ is a compact
open subgroup of $M^{0}(F)$ containing the parahoric $M^0(F)_{y,0}$
as a normal subgroup,
and $\rho_{M^{0}}$ is an irreducible smooth representation of $K_{M^{0}}$
such that $\rho_{M^{0}}|M^0(F)_{y,0}$ contains a cuspidal representation
of the finite quotient $M^{0}(F)_{y,0:0+}$.
Finally, $\vec{\phi}=(\phi_{0},\ldots,\phi_{d})$,
where each $\phi_{i}$ is a character of $G^{i}(F)$.
The datum $\Sigma$
is constrained by several conditions (loc.\ cit.). 

From the datum $\Sigma$, Kim and Yu's construction produces a sequence
of Bushnell-Kutzko types $(K^{i},\rho_{i})$ for $G^{i}(F)$. Write
$K=K^{d}$ and $\rho=\rho_{d}$. 
\begin{thm}
\label{T-types}Let $\Sigma$ be the datum as above. Assume that for
$0\leq i\leq d$, $G^{i}$ is the Levi factor of a rational
parabolic subgroup of $G$.
Moreover,
in the commutative diagram $\{\iota\}$, choose
$\BBB(G^{i},F)\hookrightarrow\BBB(G^{i+1},F)$
to be $r_{i}/2$-generic.
Then $\Sigma$ determines an isomorphism
\[
\HH(G,\rho^{\vee})\overset{\sim}{\longrightarrow}\HH(G^{0},\rho_{0}^{\vee}).
\]
\end{thm}
\begin{proof}
Let $P=G^{d-1}U$ be a parabolic subgroup of $G$ with Levi factor
$G^{d-1}$ and $\bar{P}=G^{d-1}\bar{U}$ the opposite parabolic. Then
by \cite{KY}*{Lemma 6.2(a)}, $K$ is decomposed with respect to
$(U,G^{d-1},\bar{U})$.
Let $K_{+}=K_{+}^{d}$ be as defined in \cite{KY}*{\S 7.4}. 
Then since $K\cap U=K_{+}\cap U$ and $K\cap\bar{U}=K_{+}\cap\bar{U}$
by the genericity of $\{\iota\}$, it follows that $K\cap U$, $K\cap\bar{U}$
$\subset\mathrm{ker}(\rho_{d})$. By \cite{KY}*{Theorem 8.1}, the
support of the Hecke algebra $\HH(G,\rho^{\vee})$ is contained
in $KG^{d-1}(F)K$.
Thus, by \cite{bushnell98}*{Theorem 7.2(ii)}, 
\[
\HH(G,\rho^{\vee})\overset{\sim}{\longrightarrow}\HH(G^{d-1},\rho_{d-1}^{\vee}).
\]
 The result follows by induction. 
\end{proof}
\begin{cor}\label{cor:T-types}
Let $\Sigma$ be as in equation (\ref{eq:datum}). In the commutative 
diagram $\{\iota\}$, choose $\BBB(G^{i},F)\hookrightarrow\BBB(G^{i+1},F)$ to be 
$r_{i}/2$-generic. Assume moreover that the connected center of $G^0$ is split modulo the center of $G$. 
Then $\Sigma$ determines an isomorphism
\[
\HH(G,\rho^{\vee})\overset{\sim}{\longrightarrow}\HH(G^{0},\rho_{0}^{\vee}).
\]
\end{cor}

\begin{proof}
Let $Z^i$ denote the connected center of $G^i$, $0\leq i \leq d$. Then $G^i$ is the centralizer of 
$Z^i$ in $G$. Since by assumption, $Z^i\an$ is contained in the center of $G$, 
it follows that $G^i$ is the centralizer of $Z^i\spl$ in $G$. Therefore
all $G^i$ are in fact $F$-Levi subgroups. The result then follows from Theorem \ref{T-types}.
\end{proof}

Assume $G$ is quasi-split and let $T$ be a maximal torus contained
in an $F$-Borel subgroup.
Write $T^\der$ for the preimage of
$T$ in $G^\der$.
Assume Hypothesis \ref{hyp:p}.
Let $\theta$ be a character of $T(F)$.
Let $(\vec{G},\vec{\phi})$ denote
the Kaletha-Howe factorization of
$(T,\theta)$ and let $r_i$
denote the depth of $\phi_i$.
Put $M^{0}=T$,
$K_{M^{0}}=T(F)_{0}$ and $\rho_{M^{0}}=\theta|T(F)_{0}$. Since
$T$ is contained in a Borel subgroup, $M^{i}=T$ for $0\leq i\leq d$. Choose
any point $y\in\BBB(M^{0},F)$ and choose $\{\iota\}$ such that
$\BBB(G^{i},F)\hookrightarrow\BBB(G^{i+1},F)$ is $r_{i}/2$-generic.
Let $\Sigma$ be the datum consisting of these choices and
let
$(K^{i},\rho_{i})$ be the Bushnell-Kutzko types for $G^{i}$ constructed
out of $\Sigma$ by Kim and Yu's construction. Then $(K,\rho)$ is
a $[T,\theta]_{G}$-type and $(K^{0},\rho_{0})$ is a $[T,\theta_{0}]_{G^{0}}$-type
for a depth-zero character $\theta_{0}$ of $T(F)$.

\begin{prop}
\label{prop:prin-series}
Assume that $p$ is co-prime to the order of $T\an \cap T\spl$ 
and to the order of the component group $\pi_0(T\an\cap T^\der)$.
Suppose $\theta$ is a character of $T(F)$ that is trivial
on $T\an(F)_{0+}$.
Let $\Sigma$ be a datum as in the previous paragraph.
Then $\Sigma$ can be chosen suitably so that 
there is an isomorphism
\[
\HH(G,\rho^{\vee})
\overset{\sim}{\longrightarrow}
\HH(G^{0},\rho_{0}^{\vee}),
\]
determined by this choice. 
\end{prop}

\begin{proof}
From Theorem \ref{T-types},
it will be enough to show that 
the groups $G^i$ appearing in the datum $\Sigma$
are all Levi factors of parabolic $F$-subgroups of $G$.
Equivalently, each $G^i$ is the centralizer in $G$
of an $F$-split torus.

We will see from Lemma \ref{lem:kaletha-howe}
that $\theta$ has a 
Howe factorization $(\phi_{-1},\phi_0, \ldots, \phi_d)$
where each character $\phi_i$, $i\geq 0$,
is trivial on $T\an(F)$.

Choose an additive character $\Lambda$ of $F$ that is nontrivial
on the ring of integers in $F$, but trivial on the prime ideal.

For $0\leq i < d$, let $r_i$ denote the depth of $\phi_i$.
Composing each $\phi_i|_{T(F)_{r_i}}$ with the isomorphism
$e_i\colon
\Lie(T)(F)_{r_i}/\Lie(T)(F)_{r_i+}
\longrightarrow T(F)_{r_i}/T(F)_{r_i+}$,
we obtain a character of
$\Lie(T)(F)_{r_i}/\Lie(T)(F)_{r_i+}$.
Such a character must have the form
$X \mapsto \Lambda(Y (X))$,
where $Y \in \Lie^*(T)(F)_{-r_i}/\Lie^*(T)(F)_{-r_i+}$. 
We can choose a good coset representative $a_i$ in 
$Y$ so that $G^i$ is the centraliser $C_G(a_i)(F)$ of $a_i$ in $G$. 
By our assumption on $p$, we can in fact choose $a_i$
to be in $\Lie (T\spl)(F)$. Then it follows that $G^i$ is 
in fact an $F$-Levi subgroup. Indeed, $C_G(a_i)(F)=C_G(T_{a_i})(F)$, 
where $T_{a_i}$ is the $F$-split subtorus of $T$ corresponding
to the Galois fixed co-torsion-free submodule $X_{a_i}$ of 
the character lattice $X$ of $T$, where 
$X_{a_i}=\{x\in X\mid (dx)(a_i)=0\}$ where $dx$ denotes the 
derivative of $x$. 
\end{proof}

We take care of some unfinished business from the previous proof.
Let $S\subseteq G$ be a tame maximal $F$-torus,
and
$S'\subset S$ an $F$-subtorus. Denote by $S^\der$ the 
preimage of $S$ in $G^\der$. Assume that $p$ satisfies Hypothesis
\ref{hyp:p} and is co-prime to the order of the component group $\pi_0(S'\cap S^\der)$. 
\begin{lem}
\label{lem:kaletha-howe}
Let
$\theta\colon S(F) \longrightarrow \C\mult$ be a character
of depth $r>0$
that is trivial on $S'(F)_{0+}$.
Then the pair $(S,\theta)$
has a Howe factorization $(\phi_{-1},\phi_0,\ldots,\phi_d)$,
where each character $\phi_i$, $i\geq 0$, is also trivial on $S'(F)$.
\end{lem}

\begin{proof}
The existence of a Howe factorization is 
\cite{kaletha:regular-sc}*{Proposition 3.6.7}.
For us, the key step is 
Lemma 3.6.9 \emph{loc.\ cit.},
which shows that 
if $\theta$ is $G$-generic, then there is a character $\phi$ of $G(F)$
such that $\phi$ and $\theta$ agree on $S(F)_r$.
It will be enough to show that $\phi$ can be chosen to be
trivial on $S'(F)$.

The first part of the proof of Lemma 3.6.9 \emph{loc.~cit.}
shows that $\theta$ is trivial on $S^\der(F)_r$. 
Write $D'=G/S'G^\der=S/S'S^\der$. From Lemma 3.1.3 \emph{loc.~cit.},
it follows that $D'(F)_r=S(F)_r / (S'S^\der)(F)_r$. 
Now again by  Lemma 3.1.3 \emph{loc.~cit.}, 
$(S'/{(S'\cap S^\der)})(F)_r
\cong (S'/{(S'\cap S^\der)^\circ})(F)_r
=S'(F)_r/{(S'\cap S^\der)^\circ}(F)_r$. 
Thus $\theta$ is trivial on $S'(F)_r$ and on $(S'/{(S'\cap S^\der)})(F)_r=(S'S^\der/{S'})(F)_r$ 
and therefore on $(S'S^\der)(F)_r$. Consequently, $\theta$ descends to a character of
$D'(F)_r$. This character can be extended to a character
$D'(F) \rightarrow \mathbb{C}^\times$, 
trivial on $D'(F)_{r+}$,
which can then be pulled back
to give a character $\phi$ of $G(F)$
that is trivial on $S'(F)$, and
whose restriction to $S(F)_r$ is equal to 
that of $\theta$. 
\end{proof}


Proposition \ref{prop:prin-series} implies the following.

\begin{cor}
\label{P-types}
Suppose the derived group $G^\der$ is $F$-split,
and $p$ satisfies Hypothesis \ref{hyp:p}.
Then the Hecke algebra of every principal series block of $G(F)$
is isomorphic to the Hecke algebra of a depth-zero principal series
block of a Levi subgroup $G^{0}(F)$ of $G(F)$:
\[
\HH(G,\rho^{\vee})
\overset{\sim}{\longrightarrow}
\HH(G^{0},\rho_{0}^{\vee}).
\]
\end{cor}

\begin{rem}
The Hecke algebras of depth-zero Bernstein blocks are known due to the
work of Morris \cite{Morris93}.
Therefore Corollary \ref{P-types} can be used to
produce generators and relations for the Hecke algebras of principal
series blocks of split groups. This was worked out by Roche \cite{roche:2000}.
Note that Proposition \ref{prop:prin-series}
also
applies to some principal-series blocks of non-split groups.
\end{rem}

\section{Consequences for ABPS}

\subsection{Twisted extended quotient}

We recall here the notion of twisted extended quotients as given in
\cite{abps17}*{\S2.1}. Let $\Gamma$ be a group acting on a topological
space $X$ and let $\Gamma_{x}$ denote the stabilizer in $\Gamma$ of $x\in X$.
Let $\natural$ denote a family of $2$-cocycles
\[
\natural_x \colon \Gamma_{x}\times\Gamma_{x}\longrightarrow\C\mult.
\]
Define 
\[
\widetilde{X}_{\natural}
=
\{
(x,\rho)
\mid
x\in X,\rho\in\Irr\C[\Gamma_{x},\natural_{x}]
\},
\]
where $\C[\Gamma_{x},\natural_{x}]$ denotes the group algebra
of $\Gamma_{x}$ twisted by $\natural_{x}$.
Topologize $\widetilde{X}_{\natural}$
by requiring that a subset of $\widetilde{X}_{\natural}$ is open
if and only if its projection to the first coordinate is open in $X$.
Let $\{\phi_{\gamma,x} \mid (\gamma,x)\in\Gamma\times X\}$,
denote a family of algebra isomorphisms
\[
\phi_{\gamma,x}:\C[\Gamma_{x},\natural_{x}]\longrightarrow\C[\Gamma_{\gamma x},\natural_{\gamma x}]
\]
satisfying the conditions:
\begin{enumerate}[(a)]
\item
if $\gamma x=x$, then $\phi_{\gamma,x}$ is conjugation
by an element of $\C[\Gamma_{x},\natural_{x}]\mult$. 
\item
$\phi_{\gamma',\gamma x}\phi_{\gamma,x}=\phi_{\gamma'\gamma,x}$
for all $\gamma',\gamma\in\Gamma$ and $x\in X$. 
\end{enumerate} 
Define a $\Gamma$-action on $\widetilde{X}_{\natural}$ by 
\[
\gamma\cdot(x,\rho)=(\gamma x,\rho\circ\phi_{\gamma,x}^{-1}).
\]
Then the \textit{twisted extended quotient} of $X$ by $\Gamma$
is defined to be: 
\[
(X\doubleslash \Gamma)_{\natural}
:=
\widetilde{X}_{\natural}\doubleslash\Gamma.
\]
\begin{rem}
The twisted extended quotient depends on the choices of the algebra
isomorphisms $\phi_{\gamma,x}$. If $\natural_{x}$ is trivial for
all $x\in X$, then there is a canonical choice. Namely, $\phi_{\gamma,x}$
is conjugation by $\gamma$ (see \cite{abps17}*{\S2.1}). 
\end{rem}

\subsection{ABPS Conjecture}

Write $\sss=[L,\pi]_{G}$, $\ttt=[L,\pi]_{L}$
and let $W^{\sss,t}$ denote the stabilizer in $W^\sss$
of a point $t$ in $\Irr^\ttt(L(F))$. 

The ABPS conjecture \cite{abps17}*{\S2.3} asserts that there exists
a family of 2-cocycles 
\[
\natural_{t}:W^{\sss,t}\times W^{\sss,t}\longrightarrow\C\mult,
\: t\in\Irr^\ttt(L(F))
\]
such that there is a natural bijection 
\begin{equation}
\Irr^\sss(G(F))\longleftrightarrow(\Irr^\ttt(L(F))\doubleslash W^\sss)_{\natural}.
\end{equation}
This bijection is expected to be compatible with the local Langlands
correspondence (loc.\ cit.).
If $G$ is quasi-split,
the cocycles in the family $\natural$ are expected to be trivial. 

\subsection{An isomorphism for extended quotients} \label{subsec:ABPS iso}

Let $\sss_{0}$ and $\ttt_{0}$ be as in \S\ref{sec:B-center}.
Let $\iota_{\pi}$ be as in Theorem \ref{thm:weyl-isom} and $\fff$
as in Theorem \ref{mm19}. Then for $t\in\Irr^\ttt(L(F))$, since $\fff$ is equivariant with respect to $\iota_{\pi}$ by 
Theorem \ref{thm:weyl-isom}, we have
\[
\iota_{\pi}|_{W^{\sss,t}}
\colon
W^{\sss,t}\overset{\sim}{\longrightarrow}W^{\sss_{0},\fff(t)}.
\]
Each 2-cocycle $\natural_{t}$, $t\in\Irr^\ttt(L(F))$,
therefore defines a 2-cocycle 
\[
\natural_{\fff(t)}^{0}
\colon
W^{\sss_{0},\fff(t)}\times W^{\sss_{0},\fff(t)}\longrightarrow\C\mult.
\]
The following theorem is then immediate from Theorem \ref{thm:weyl-isom}
and \cite{mishra19}*{Theorem 6.1}. 

\begin{thm} \label{thm:ABPS}
Suppose Hypothesis \ref{hyp:p}.
Then there is an isomorphism 
\[
\mathfrak{l}_{\pi}
\colon
(\Irr^\ttt(L(F))\doubleslash W^\sss)_{\natural}
\overset{\sim}{\longrightarrow}
(\Irr^{\ttt_{0}}(L^{0}(F))\doubleslash W^{\sss_{0}})_{\natural^{0}}
\]
determined by the choice of $(S,\theta)$. 
\end{thm}

Let $(K,\rho)$
(resp.\ $(K^{0},\rho_{0})$ be the type constructed by Kim and Yu for
$\Rep^{\sss}(G(F))$
(resp.\ $\Rep^{\sss_{0}}(G^{0}(F))$).
We have equivalences of categories:
\[
\Rep^{\sss}(G(F))
\overset{\sim}{\longrightarrow}
\HH(G,\rho^\vee)\mathrm{-Mod},
\qquad
\Rep^{\sss_{0}}(G^{0}(F))
\overset{\sim}{\longrightarrow}
\HH(G^{0},\rho_{0}^\vee)\mathrm{-Mod}.
\]


\begin{cor}\label{cor-hyp}
Suppose Hypothesis \ref{hyp:p}. Then for regular Bernstein blocks, Conjecture \ref{conj-Hecke} implies the ABPS Conjecture. 
\end{cor}
\begin{proof}
For depth-zero Bernstein blocks, the ABPS conjecture holds by \cite{solleveld2012}
and Morris's presentation of Hecke algebras of depth-zero blocks \cite{Morris93}. The
result then follows from Theorem \ref{thm:ABPS}.
\end{proof}

\begin{thm}\label{thm:ABPS-Levi}
Assume Hypothesis \ref{hyp:p}.
Assume 
that the twisted Levi subgroups $G^{0}\subsetneq\ldots\subsetneq G^{d}=G$
obtained in the Howe factorization of $(S,\theta)$ in $G$ are 
$F$-Levi subgroups of $G$.
Then the ABPS conjecture holds for the Bernstein
block $\Rep^{\sss}(G(F))$.
In particular, if the connected center of $G^0$ is split modulo the center of $G$,
then the ABPS Conjecture 
holds.
Therefore it also holds for principal series blocks of split groups.
\end{thm}

\begin{proof}
The first claim follows from Theorem \ref{T-types} and Corollary
\ref{cor-hyp}.
The last two claims follow from
Corollaries \ref{cor:T-types}, \ref{P-types} and \ref{cor-hyp}.
\end{proof}

\begin{rem}

When $G$ is a classical group, Equation (\ref{eq-Hecke}) for
certain Bernstein blocks can be observed from results of
Kim \cite{Kim2001}.
More generally,
for classical groups, Heiermann \cite{Heiermann2017}
has established certain equivalences
of categories in the spirit of Equation (\ref{eq-R_f}). It would 
be worthwhile to understand the equivalence in terms of Kim-Yu data. 
\end{rem}

\section{Partial results toward further Hecke algebra isomorphisms}
\label{sec:partial}
We prove weaker versions of
the Hecke algebra isomorphism of Equation (\ref{eq-Hecke}).
Let the notation be as in \S\ref{sec:B-center} and \S\ref{subsec:ABPS iso}. 

The equivalence 
\[
\Rep^{\sss}(G(F))
\overset{\sim}{\longrightarrow}
\HH(G,\rho^\vee)\mathrm{-Mod},
\]
induces an isomorphism of $\ZZ^{\sss}$
with the center of $\HH(G,\rho^\vee)$.
View $\HH(G,\rho^\vee)$
as a $\ZZ^{\sss}$-algebra.
It is then finitely generated as a $\ZZ^{\sss}$-module.
Similarly, 
$\HH(G^{0},\rho_{0}^\vee)$ is a finitely generated module over its center
$\ZZ^{\sss_{0}}$.

Write $S=\ZZ^{\sss}\backslash0$
(resp.\ $S_{0}=\ZZ^{\sss_{0}}\backslash0$)
and let $\mathrm{k}_{\sss}:=S^{-1}\ZZ^{\sss}$
(resp.\ $\mathrm{k}_{\sss_{0}}:=S_{0}^{-1}\ZZ^{\sss_{0}}$)
denote the field of fractions of $\ZZ^{\sss}$
(resp.\ $\ZZ^{\sss_{0}}$). Write 
\[
\mathrm{H}(G,\rho^\vee):=S^{-1}\HH(G,\rho^\vee).
\]
 Similarly, 
\[
\mathrm{H}(G^{0},\rho_{0}^\vee):=S_{0}^{-1}\HH(G^{0},\rho_{0}^\vee).
\]
 
In \cite{BH03}*{\S 4.3}, Bushnell and Henniart define the notions of
\textit{generic}
and \textit{simply generic}, which are generalizations
of the generic representations in quasi-split groups. 
 
\begin{thm}
Assume Hypothesis \ref{hyp:p}.
Assume that $\pi|_{{}^{\circ}L(F)}$ and $\pi_{0}|_{{}^{\circ}L^{0}(F)}$
are multiplicity free. Then there is a vector-space isomorphism 
\[
\HH(G,\rho^\vee)
\overset{\sim}{\longrightarrow}
\HH(G^{0},\rho_{0}^\vee).
\]
If $\sss$ and $\sss_{0}$ are simply generic,
then there is a $\mathbb{C}$-algebra isomorphism 
\[
\mathrm{H}(G,\rho^\vee)
\overset{\sim}{\longrightarrow}
\mathrm{H}(G^{0},\rho^\vee_{0}).
\]
 
\end{thm}
\begin{proof}
By \cite{roche:BC}*{Prop. 1.8.4.1}, $\HH(G,\rho^\vee)$
(resp.\  $\HH(G^{0},\rho^\vee_{0})$)
is a free right-module over $\ZZ^{\ttt}$
(resp.\ $\ZZ^{\ttt_{0}}$)
of rank $|W^{\sss}|$
(resp.\ $|W^{\sss_{0}}|$).
By \cite{mishra19}*{Theorem 6.1},
$\ZZ^{\ttt}\cong\ZZ^{\ttt_{0}}$
and by Theorem \ref{thm:weyl-isom},
$W^{\sss}\cong W^{\sss_{0}}$.
The first claim is then immediate. 

By \cite{BH03}*{Theorem 5.2},
\[
\mathrm{H}(G,\rho^\vee)\cong\mathrm{M}_{n}(\mathrm{k}_{\sss}),
\]
for some integer $n$. Similarly
\[
\mathrm{H}(G^{0},\rho^\vee_{0})\cong\mathrm{M}_{n_{0}}(\mathrm{k}_{\sss_{0}}),
\]
for some integer $n_{0}$.
Since $\mathrm{k}_{\sss}\cong\mathrm{k}_{\sss_{0}}$
by Theorem \ref{thm:b-center-isom},
to prove
$\mathrm{H}(G,\rho^\vee)\overset{\sim}{\rightarrow}\mathrm{H}(G^{0},\rho^\vee_{0})$,
it suffices to prove 
\[
\dim_{\mathrm{k}_{\sss}}\mathrm{H}(G,\rho^\vee)=\dim_{\mathrm{k}_{\sss_{0}}}\mathrm{H}(G^{0},\rho^\vee_{0}).
\]
Now $S^{-1}\ZZ^{\ttt}\cong S_{0}^{-1}\ZZ^{\ttt_{0}}$
by Theorem \ref{thm:b-center-isom} and \cite{mishra19}*{Theorem 6.1}.
The claim then follows from \cite{roche:BC}*{Prop. 1.8.4.1} .
\end{proof}

\end{document}